\documentclass[10pt,a4paper,reqno]{amsart}

\usepackage{amsfonts}
\usepackage[margin=2cm]{geometry}
\usepackage{graphics, amsmath,enumitem, comment, color}
\usepackage{graphicx}
\usepackage{epsfig}
\usepackage{amssymb}
\usepackage{amsmath,amsthm}
\usepackage{mathrsfs}
\usepackage{bbm}
\usepackage{colortbl}

\newif\ifcolorcomments
\newcommand{\allowcomments}[4]{
\newcommand{#1}[1]{\ifdraft{\ifcolorcomments{\textcolor{#4}{##1 --#3}}\else{\textsl{ ##1 \ --#3}}\fi}\else{}\fi}
}
\allowcomments{\comDS}{DS}{David}{green}
\allowcomments{\comAB}{AB}{Ayreena}{magenta}
\colorcommentstrue
\def\bc{\begin{center}}
\def\ec{\end{center}}
\def\be{\begin{equation}}
\def\ee{\end{equation}}

\def\N{\mathbb N}
\def\Z{\mathbb Z}

\def\R{\mathbb R}

\newcommand{\0}{\mathbf 0}
\newcommand{\qq}{\mathbf q}
\newcommand{\pp}{\mathbf p}
\newcommand{\bb}{\mathbf b}
\newcommand{\xx}{\mathbf x}
\newcommand{\yy}{\mathbf y}
\newcommand{\vv}{\mathbf v}
\newcommand{\ww}{\mathbf w}
\newcommand{\zz}{\mathbf z}
\newcommand{\eee}{\mathbf e}

\def\Z{\mathbb Z}

\newtheorem{lem}{Lemma}[section]

\newtheorem{dfn}[lem]{Definition}
\newtheorem{pro}[lem]{Proposition}
\newtheorem{thm}[lem]{Theorem}

\newtheorem{cor}[lem]{Corollary}

\numberwithin{equation}{section}

\newif\ifdraft\drafttrue

\allowdisplaybreaks
\usepackage{hyperref}
\begin{document}

 \newpage
\title[Hausdorff measure and Dirichlet's Theorem in higher dimensions]{Generalised Hausdorff measure of  sets of Dirichlet non-improvable matrices in higher dimensions}
\author[A. Bakhtawar]{Ayreena ~Bakhtawar}
\address{School of Mathematics and Statistics, University of New South Wales, Sydney, NSW 2052, Australia}
\email{A Bakhtawar: A.Bakhtawar@unsw.edu.au}
\author[D. Simmons]{David~Simmons}
\address{University of York, Department of Mathematics, Heslington, York, YO10 5DD, UK}
\email{D Simmons: David.Simmons@york.ac.uk}

\begin{abstract}
Let $\psi:\mathbb R_{+}\to \mathbb R_{+}$ be a non-increasing function. A pair $(A,\mathbf b),$ where $A$ is a real $m\times n$ matrix and $\mathbf b\in\mathbb R^{m},$ is said to be $\psi$-Dirichlet improvable, if the system
$$\|A\mathbf q +\mathbf b-\mathbf p\|^m<\psi(T), \quad \|\mathbf q\|^n<T$$
is solvable in $\mathbf p\in\mathbb Z^{m},$ $\mathbf q\in\mathbb Z^{n}$
for all sufficiently large $T$ where $\|\cdot\|$ denotes the supremum norm. For $\psi$-Dirichlet non-improvable sets, Kleinbock--Wadleigh (2019) proved the Lebesgue measure criterion whereas Kim--Kim (2022) established the Hausdorff measure results. In this paper we obtain the generalised Hausdorff $f$-measure version of Kim--Kim (2022) results for $\psi$-Dirichlet non-improvable sets.
\end{abstract}

\maketitle

\section{Introduction}
 To begin with, we recall the higher dimensional general form of Dirichlet's Theorem (1842). 
 \noindent Let $m, n$  be positive integers and let $X^{mn}$ denotes the space of real $m \times n$ matrices.
\begin{thm}[Dirichlet's Theorem]\label{HigDir}
  Given any $A~ \in~X^{mn}$ and $T>1,$ there exist $\pp \in \Z^{m}$ and $\qq \in \Z^n \setminus  \{0\}$
   such that
  \begin{equation}\label{HDirlinearform}
    \|A\qq -\pp\|^m \le \frac{1}{T} \text{ and }  \|\qq\|^n<T.
  \end{equation}
\end{thm}
Here $\|\cdot\|$ denotes the supremum norm in $\mathbb R^{i},$ $i\in\N.$ 
Theorem \ref{HigDir} guarantees a nontrivial integer solution for all $T.$ 
The standard application of \eqref{HDirlinearform} is the following corollary, guaranteeing that such a system is solvable for an unbounded set of $T.$ 
\begin{cor}For any $A \in
   X^{mn}$ there exist infinitely many integer vectors $\qq \in \mathbb Z^n$
  such that
  \begin{equation}\label{AS2.11}
   \|A\qq -\pp\|^m \le \frac{1}{\|\qq \|^{n}}  \text{ for some } \pp \in \mathbb Z^{m}.
     \end{equation}
\end{cor}

The two statements above give rise to two possible ways to pose Diophantine approximation problems sometimes referred to as uniform vs asymptotic approximation results: that is, looking for solvability of inequalities for all large enough $T$ vs. for some arbitrarily large $T.$ The rate of approximation given in above two statements works for all real matrices $A\in X^{mn},$  which serves as the beginning of the \textit{metric theory of Diophantine approximation,} a field concerned with understanding sets of $A\in X^{mn}$ satisfying similar conclusions but with the right hand sides replaced by faster decaying functions of $T$ and $\|\qq\|^{n}$ respectively. Those sets are well studied in the asymptotic setup \eqref{AS2.11} long ago.

Indeed, for a function $\psi:\R_{+}\to \R_{+}$ a matrix $A\in X^{mn}$ is said to be $\psi$-approximable if the inequality \footnote{\label{FT1}Here we use the definition as in \cite{KlWa_19,KlMa_99}, whereas in Section \ref{Sec4} we will consider slightly different definition such as in \cite{BeVe_010} where instead of \eqref{AS2.1} the inequality $\|A\qq-\bb\|<\psi(\|\qq\|)$ is used.}
  \begin{equation}\label{AS2.1}
   \|A\qq -\pp\|^m < \psi({\|\qq \|^{n}})  \text{ for some } \pp \in \mathbb Z^{m}
     \end{equation}
is satisfied for infinitely many integer vectors $\qq \in \mathbb Z^n.$ As the set of $\psi$-approximable matrices is translation invariant under integer vectors, we can restrict attention to $mn$-dimensional unit cube $[0,1]^{mn}.$
Then the set of $\psi$-approximable matrices in $[0,1]^{mn}$ will be denoted by
$W_{m,n}(\psi).$

The following result gives the size of the set $W_{m,n}(\psi )$ in terms of Lebesgue measure.

\begin{thm}[Khintchine--Groshev Theorem, \cite{Gr_41}] Given a non-increasing $\psi$, the set $W_{m, n}(\psi)$ has zero (respectively full) Lebesgue measure if and only if the series $\sum_k\psi(k)$ converges (respectively, diverges).
\end{thm}

Let us now briefly describe what is known in the setting of \eqref{HDirlinearform}. For a non-increasing function $\psi:[T_{0},\infty)\to \R_{+}$ with $T_{0}>1$ fixed, consider the set $D_{m,n}(\psi )$ of $\psi$-Dirichlet improvable matrices consisting of $A\in X^{mn}$ such that the system
\begin{equation*} 
  \|A\qq -\pp\|^m \le \psi{(T)} \text{ and }  \|\qq\|^n<T                
\end{equation*}
has a
nontrivial integer solution for all large enough $T$. Elements of the complementary set, $D_{m,n}(\psi )^{c},$ will be referred as $\psi$-Dirichlet non-improvable matrices.

With the notation $\psi_{a}(x):=x^{-a},$ \eqref{HDirlinearform} implies that $D_{1, 1}(\psi_{1})=\R,$ and that for any $m,n$ every matrix is $\psi_{1}$-Dirichlet improvable. It was observed in  \cite{DaSc_70} that for $\min(m,n)=1$ and in \cite{KlWe_08} for the general case, that the Lebesgue measure of $D_{m, n}(c\psi_{1})$ of the set $c\psi_{1}$-Dirichlet improvable matrices is zero for any $c<1,$ . 

The theory of inhomogeneous Diophantine approximation starts by replacing the values of a system of linear forms $A\qq$ by those of a system of affine forms $\qq \mapsto A\qq+\bb$ where  $A\in X^{mn}$ and $\bb\in\R^{m}.$ Following \cite{KlWa_18}, for a non-increasing function $\psi:[T_{0},\infty)\to \R_{+} $ a pair $(A,\bb)\in X^{mn}\times\R^{m}$ is called $\psi$-Dirichlet improvable if for all $T$ large enough, one can find nonzero integer vectors $\qq\in\Z^{n}$ and $\pp\in\Z^{m}$ such that
\begin{equation}\label{InhDT}
\|A\qq +\bb-\pp\|^m<\psi(T)\quad \text{and}\quad \|\qq\|^n<T.
\end{equation}
Let $\widehat{D}_{m, n}(\psi)$ denote the set of $\psi$-Dirichlet improvable pairs in the unit cube $[0,1]^{mn+m}$. If the inhomogeneous vector $\bb\in\R^m$ is fixed then let $\widehat{D}^{\bb}_{m, n}(\psi)$ be the set of all $A\in X^{mn}$ such that \eqref{InhDT} holds i.e. for a fixed $\bb\in\R^m$ we have $\widehat{D}^\bb_{m, n}(\psi)=\{  A\in X^{mn} :(A,\bb)\in \widehat{D}_{m, n}(\psi)\}.$   

The Lebesgue measure criterion for the set $\widehat{D}_{m, n}(\psi)$ i.e. doubly metric case  has been proved by Kleinbock--Wadleigh \cite{KlWa_19} by reducing the problem to the shrinking target problem on the space of grids in $\R^{m+n}$. The proof of their theorem is based on a correspondence between Diophantine approximation and homogenous dynamics.

\begin{thm}[Kleinbock--Wadleigh, \cite{KlWa_19}]\label{ASKW2} Given a non-increasing $\psi$, the set $\widehat{D}_{m, n}(\psi)$ has zero (respectively full) Lebesgue measure if and only if the series $\sum_j\frac{1}{j^2\psi(j)}$ diverges (respectively converges).
\end{thm}
Recently (2022), Kim--Kim \cite{KiKi_20} established the Hausdorff measure analogue of Theorem \ref{ASKW2}.
\begin{thm}[Kim--Kim, \cite{KiKi_20}] \label{KiKi1}Let $\psi$ be non-increasing with $\lim_{T\to\infty}\psi(T)=0$ and $0 \leq s\leq mn+m.$ Then
\begin{equation*} 
\mathcal{H}^{s}(  \widehat{D}_{m, n}(\psi)^{c}       )=
\begin{cases}
0\  & \text{ if }\quad \sum\limits_{q=1}^{\infty} \frac{1}{\psi(q)q^{2}} \left( \frac{ q^{\frac{1}{n}}}{\psi(q)^\frac{1}{m}}\right)^{mn+m-s}
\,<\,\infty ; \\[2ex] 
\mathcal{H}^{s}([0,1]^{mn+m}) \  & \text{ if }\quad \sum\limits_{q=1}^{\infty} \frac{1}{\psi(q)q^{2}} \left( \frac{ q^{\frac{1}{n}}}{\psi(q)^\frac{1}{m}}\right)^{mn+m-s}
\,=\,\infty .
\end{cases}
\end{equation*}
\end{thm}

 In the same article Kim--Kim also provided the Hausdorff measure criterion for the singly metric case. 
\begin{thm}[Kim--Kim, \cite{KiKi_20}] \label{KiKi2}Let $\psi$ be non-increasing with $\lim_{T\to\infty}\psi(T)=0.$ Then for any $0\leq s\leq mn$ 
\begin{equation*} 
\mathcal{H}^{s}( \widehat{D}^{\bb}_{m, n}(\psi) ^{c}      )=
\begin{cases}
0\  & \text{ if }\quad \sum\limits_{q=1}^{\infty} \frac{1}{\psi(q)q^{2}} \left( \frac{ q^{\frac{1}{n}}}{\psi(q)^\frac{1}{m}}\right)^{mn-s}
\,<\,\infty ; \\[2ex] 
\mathcal{H}^{s}([0,1]^{mn})  & \text{ if } \quad \sum\limits_{q=1}^{\infty} \frac{1}{\psi(q)q^{2}} \left( \frac{ q^{\frac{1}{n}}}{\psi(q)^\frac{1}{m}}\right)^{mn-s}
\,=\,\infty,
\end{cases}
\end{equation*}
for every $\bb\in \R^{m}\setminus \Z^{m}.$
\end{thm}

Naturally one can ask about the generalization of Theorem \ref{KiKi1} and Theorem \ref{KiKi2} in terms of $f$-dimensional Hausdorff measure. Recall that a natural generalization of the $s$-dimensional Hausdorff measure $\mathcal{H}^{s}$ is the $f$-dimensional Hausdorff measure $\mathcal{H}^{f}$ where $f$ is a dimension function, that is an increasing, continuous function $f:\R_{+}\to\R_{+}$ such that $f(r)\to 0$ as $r\to 0.$

In this article we extend the results of Kim--Kim \cite{KiKi_20} by establishing the zero-full law for the sets $\widehat{D}_{m,n}(\psi)$ and $\widehat{D}^{\bb}_{m,n}(\psi)$ in terms of generalised $f$-dimensional Hausdorff measure.
We obtain the following main results.
{\begin{thm}\label{BST1} Let $\psi$ be non-increasing and $f$ be a dimension function with
\begin{equation}\label{C1}f(xy)\asymp x^{s}f(y) \quad \forall \quad y^{\alpha}\leq x\leq y^{\frac{1}{\alpha}}\end{equation} 
where $mn+m-n<s<mn+m$ and $\alpha>1$ is some absolute constant independent of $x$ and $y$ and  suppose that  
\begin{equation}\label{C2}
f'(x)= a(x){\frac{f(x)}{x}}
\end{equation}
such that $a(x)\to s$ as $x\to0$. Further, let

\begin{equation}\label{C05}
(q^{-\frac{m}{n}})^\alpha\leq \psi(q)\leq (q^{-\frac{m}{n}})^{\frac{1}{\alpha}}.
\end{equation}

 
Then
\begin{equation*} 
\mathcal{H}^{f}(  \widehat{D}_{m, n}(\psi) ^{c}      )=
\begin{cases}
0\  & \text{ if }\quad \sum\limits_{q=1}^{\infty} \frac{1}{\psi(q)q^{2}} \left( \frac{ q^{\frac{1}{n}}}{\psi(q)^\frac{1}{m}}\right)^{mn+m}     f\left( \frac{\psi(q)^\frac{1}{m}}{ q^{\frac{1}{n}}}\right)    
\,<\,\infty ; \\[2ex] 
\mathcal{H}^{f}([0,1]^{mn+m}) \  & \text{ if }\quad \sum\limits_{q=1}^{\infty} \frac{1}{\psi(q)q^{2}} \left( \frac{ q^{\frac{1}{n}}}{\psi(q)^\frac{1}{m}}\right)^{mn+m} f\left( \frac{\psi(q)^\frac{1}{m}}{ q^{\frac{1}{n}}}\right)    \,=\,\infty .
\end{cases}
\end{equation*}
\end{thm}

For the singly metric case we have the following result.
\begin{thm}\label{BST2} Let $\psi$ be non-increasing and $f$ be a dimension function such that $r^{-mn} f(r) \to \infty$ as $r\to 0.$ Suppose that \eqref{C1} -- \eqref{C05} holds and $mn-n<s<mn.$ 
Then
\begin{equation*} 
\mathcal{H}^{f}(\widehat{D}^{\bb}_{m, n}(\psi)^{c})=
\begin{cases}
0\  & \text{ if }\quad \sum\limits_{q=1}^{\infty} \frac{1}{\psi(q)q^{2}} \left( \frac{ q^{\frac{1}{n}}}{\psi(q)^\frac{1}{m}}\right)^{mn}     f\left( \frac{\psi(q)^\frac{1}{m}}{ q^{\frac{1}{n}}}\right)    
\,<\,\infty ; \\[2ex] 
\mathcal{H}^{f}([0,1]^{mn}) \  & \text{ if }\quad \sum\limits_{q=1}^{\infty} \frac{1}{\psi(q)q^{2}} \left( \frac{ q^{\frac{1}{n}}}{\psi(q)^\frac{1}{m}}\right)^{mn} f\left( \frac{\psi(q)^\frac{1}{m}}{ q^{\frac{1}{n}}}\right)    \,=\,\infty 
\end{cases}
\end{equation*}
for every $\bb\in \R^{m}\setminus \Z^{m}.$
\end{thm}

We remark that the conditions \eqref{C1} and \eqref{C2} are satisfied in a wide variety of cases, for example $f(x) = x^s \log^t(x)$ for some $s > 0$ and $t\in\R$. Indeed, \eqref{C1} follows since $f(xy) = (xy)^s \log^t(xy) \asymp x^s y^s \log^t(y) = x^s f(y)$, and \eqref{C2} follows since
$$
\frac{x f'(x)}{f(x)} = x \frac{d}{dx}[s\log(x) + t \log\log(x)] = x\left(\frac {s}{x} + \frac t{x\log(x)}\right) \to s \text{ as }{x\to 0}.
$$

\noindent\textbf{Acknowledgements.}
{The research of first-named author was supported by the Australian Mathematical Society Lift-off Fellowship and  the Australian Research Council Grant
DP180100201. The second-named author was supported by a Royal Society University Research Fellowship, URF$\backslash$R1$\backslash$180649. The authors would like to thank the anonymous reviewer for the careful reading of the paper which has led to multiple improvements.}

\section{Preliminaries and auxiliary results}
\subsection{ Hausdorff measure and dimension}
Let $f:~{\R}_+\rightarrow~{\R}_+$  be a dimension function i.e. an increasing continuous function such that $f(r)\to 0$ as $r\to~0$ and let $\mathcal V$ be an arbitrary subset of $\R^n.$ For $\rho>0,$ a  $\rho$-cover for a set $\mathcal V$ is defined as a countable collection $\{U_i\}_{i\geq1}$ of sets in
$\R^n$ with diameters $0<\mathrm{diam} (U_i)\le \rho$ such that
$\mathcal V\subseteq \bigcup_{i=1}^{\infty} U_i$. 
Then for each $\rho>0$ define
\begin{equation*}
\mathcal H_{\rho}^f(\mathcal V)=\inf \left\{\sum_{i=1}^{\infty} f\big(\mathrm{diam}(U_i)\big):    \{U_{i}\}   {\text { is a $\rho$ -cover of }} \mathcal V      \right\}.    \end{equation*}
 Note that 
 $\mathcal H_\rho^f(\mathcal V)$ is 
 non-decreasing as $\rho$ decreases and therefore approaches a limit as $\rho\to0$. Accordingly, the $f$-dimensional Hausdorff measure of $\mathcal V$ is defined as
 \begin{equation*} 
\mathcal H^f(\mathcal V):=\lim_{\rho\to 0}\mathcal H_\rho^f(\mathcal V).
\end{equation*}
This limit could be zero or infinity, or take a finite positive value. 

 If $f(r)=r^s$ where $s > 0$,  then $\mathcal H^{f}$ is the $s$-dimensional Hausdorff measure and is represented by  $\mathcal H^s.$ 
It can be easily verified that Hausdorff measure is monotonic, that is, if $E$ is contained in $F$ then $\mathcal H^{s}(E)\leq \mathcal H^{s}(F)$, countably sub-additive, and satisfies $\mathcal H^s(\emptyset)=0$. 

The following property 
\begin{equation*}
\mathcal H^{s}(\mathcal V)<\infty \implies \mathcal H^{s'}(\mathcal V)=0 \quad {\text{ if }} s'>s,
\end{equation*}
implies that there is a unique real point $s$ at which the Hasudorff $s$-measure drops from infinity to zero (unless $\mathcal V $ is finite so that $\mathcal H^{s}(\mathcal V)$ is never infinite). 
The value taken by $s$ at this discontinuity is referred to as the \textit{Hausdorff dimension} of a set $\mathcal V$ and is defined as

\begin{equation*}
\dim _{\mathrm{H}} \mathcal V:=\inf\{s > 0:\; \mathcal H^s(\mathcal V)=0\}.
\end{equation*}

For establishing the convergent part of Theorem \ref{BST1} and Theorem \ref{BST2} we will apply the following Hausdorff measure version of the famous Borel--Cantelli lemma \cite[Lemma 3.10]{BeDo_99}:

\begin{lem}\label{CB}
Let $\{B_{i}\}_{i\geq1}$ be a sequence of measurable sets in $\R^{n}$ and suppose that for some dimension function $f,$ $\sum_{i}f(\mathrm {diam}(B_{i}))<\infty.$ Then
$\mathcal{H}^{f}( \limsup_{i\to\infty}  B_{i}    )=0.$
\end{lem}

We will use the following principle known as Mass Distribution Principle \cite[\S 4.1]{Fa_14} for the divergent part of Theorem \ref{BST1}.

\begin{lem}\label{MDP} Let $\mu$ be a probability measure supported on a subset $\mathcal V$ of $\R^{k}.$ Suppose there are positive constants $c>0$ and $\varepsilon>0$ such that 
\begin{equation*}
\mu(U)\leq c f(\mathrm {diam}(U))
\end{equation*} 
for all sets $U$ with $\mathrm{diam}(U)\leq\varepsilon.$ Then $\mathcal{H}^{f}(\mathcal V)\geq \mu(\mathcal V)/c.$\end{lem}

\begin{thm}[{\cite[Theorem 2]{AlBe_18}}]\label{AlBeJT}
Let $\psi:\N \to \R_{+}$ be any approximating function and let $mn>1.$ Let $f$ and $g:r\to g(r):=r^{-m(n-1)}f(r)$ be dimension functions such that $r \mapsto  r^{-mn}f(r)$ is monotonic. Then
\begin{equation*} 
\mathcal{H}^{f}(  {W}_{m, n}(\psi)      )=
\begin{cases}
0\  & \text{ if }\quad \sum\limits_{q=1}^{\infty} q^{m+n-1}g\left( \frac{\widehat{\psi}(q)}{q}\right)
\,<\,\infty ; \\[2ex] 
\mathcal{H}^{f}([0,1]^{mn})  & \text{ if } \quad \sum\limits_{q=1}^{\infty} q^{m+n-1} g\left( \frac{\widehat{\psi}(q)}{q}\right)\,=\,\infty,
\end{cases}
\end{equation*}
where $\widehat{\psi}(q)=\psi(q^{n})^{\frac{1}{m}}.$

\end{thm}

\subsection{Ubiquitous systems}\label{US}
To prove the divergent parts of Theorem \ref{BST2} we will use the ubiquity technique developed by Beresnevich, Dickinson, and Velani, see \cite[\S12.1]{BeDiVe_06}. The idea and concept of ubiquity was originally formulated by Dodson, Rynne, and Vickers in \cite{DoRyVi_91} and coincided in part with the concept of `regular systems' of Baker and Schmidt \cite{BaSc_70}. Both have proven to be extremely useful in obtaining lower bounds for the Hausdorff dimension of limsup sets. The ubiquity framework in \cite{BeDiVe_06} provides a general and abstract approach for establishing the Lebesgue and Hausdorff measure of a large class of limsup sets.

Consider the $mn$-dimensional unit cube $[0,1]^{mn}$ with the supremum norm $\|\cdot\|$. Let $\mathcal{R}=\{ R_{\kappa}\subseteq [0,1]^{mn}:\kappa\in J\}$ be a family of subsets, referred to as resonant sets $R_{\kappa}$ of $[0,1]^{mn}$ indexed by an infinite, countable set $J.$ Let $\beta:J\to\R_{+}:\kappa\mapsto \beta_{\kappa}$ be a positive function on $J$ i.e. the function $\beta$ attaches the weight $\beta_{\kappa}$ to the set $R_{\kappa}.$ Next assume that the number of terms $\kappa$ in $J$ with $\beta_{\kappa}$ bounded above is always finite. Following the ideas from \cite[\S12.1]{BeDiVe_06} and \cite{KiKi_20} let us assume that the family $\mathcal R$ of resonant sets $R_{\kappa}$ consists of $(m-1)n$-dimensional, rational hyperplanes and define the following notations.
 For a set $S\subseteq [0,1]^{mn},$ let 
\begin{equation*}
\Delta(S,r):=\{V\in[0,1]^{mn}:\text{dist}(V,S)<r\},
\end{equation*}
where $\text{dist}(V,S):=\inf\{ \|V-Y \|:Y\in S \}.$ Fix a decreasing function $\Psi:\R_{+}\to\R_{+}$  let

\begin{equation}
\Lambda(\Psi)=\{ V\in[0,1]^{mn}:V \in \Delta(R_{\kappa},\Psi(\beta_{\kappa})) \text{ for i.m. } \kappa\in J \}
\end{equation}
The set $\Lambda(\Psi)$ is a lim sup set; it consists of elements of $[0,1]^{mn}$ which lie in infinitely many of the thickenings $\Delta(R_{\kappa},\Psi(\beta_{\kappa})).$ It is natural to call $\Psi$ the approximating function as it governs the `rate' at which the elements of $[0,1]^{mn}$ must be approximated by resonant sets in order to lie in $\Lambda(\Psi)$. Let us rewrite the set $\Lambda(\Psi)$ in a way which brings its lim sup nature to the forefront.

For $N\in\N,$ let 
\begin{equation*}
\Delta(\Psi,N):=\bigcup_{\kappa\in J:2^{N-1}<\beta_{\kappa}\leq 2^{N}}\Delta(R_{\kappa},\Psi(\beta_{\kappa})).
\end{equation*}
Thus $\Lambda(\Psi)$ is the set consisting elements of $[0,1]^{mn}$ which lie in infinitely many $\Delta(\Psi,N),$ that is,
\begin{equation}
\Lambda(\Psi):=\limsup_{N\to\infty}\Delta(\Psi,N)
\end{equation}
Next let $\rho:\R_{+} \to \R_{+}$ be a function with $\rho(t)\to0$ as $t\to\infty$ and let 
\begin{equation}\Delta(\rho,N):=\bigcup_{\kappa\in J:2^{N-1}<\beta_{\kappa}\leq 2^{N}}\Delta(R_{\kappa},\rho(\beta_{\kappa})).
\end{equation}

\begin{dfn}
Let B be an arbitrary ball in $[0,1]^{mn}.$ Suppose there exist a function $\rho$ and an absolute constant $\kappa>0$ such that
\begin{equation}
|B\cap\Delta(\rho,N)|\geq\kappa|B| \text{ for } N\geq N_{0}(B),
\end{equation}
where $|\cdot|$ denotes the Lebesgue measure on $[0,1]^{mn}.$ Then the pair $(\mathcal R,\beta)$ is said to be a `local ubiquitous system' relative to $\rho$ and the function $\rho$ will be referred to as the `ubiquitous function'.
\end{dfn}

A function $h$ is said to be $2$-regular if there exists a strictly positive constant $\lambda<1$ such that for $N$ sufficiently large
$$h(2^{N+1})\leq\lambda h(2^{N}).$$

The next theorem is a simplified version of Theorem 1 and Theorem 2 from \cite{BeDiVe_06}. To state the result we define notions similar to those in \cite{BeDiVe_06}. Note that with notions in \cite{BeDiVe_06}, we have $\Omega:=[0,1]^{mn}$, the Lebesgue measure on $[0,1]^{mn}$ is of type (M2) with $\delta=mn$ and $\gamma=(m-1)n$ and the local ubiquitous system $(\mathcal R,\beta)$ satisfies the intersection conditions with $\gamma=(m-1)n$ (see \cite[section 12.1]{BeDiVe_06}).
 Given that the Lebesgue measure is comparable with $\mathcal{H}^{\delta}-$ a simple consequence of (M2), we have the following combined version of Theorem 1 and Theorem 2 from \cite{BeDiVe_06}. 
 \begin{thm}\label{thm5}
 Suppose that $(\mathcal R,\beta)$ is a local ubiquitous system relative to $\rho$ and that $\Psi$ is an approximating function. Let $f$ be a dimension function such that $r^{-nm}f(r)$ is monotonic, $r^{-nm}f(r)\to\infty$ as $r\to 0$ and $r^{-n(m-1)}f(r)$ is increasing.
Furthermore, suppose that $\rho$ is $2$-regular and 
\begin{equation}\label{divcon}
\sum^{\infty}_{n=1}\frac{(\Psi(2^{N}))^{-n(m-1)}f(\Psi(2^{N}))}{\rho(2^{N})^{n}}=\infty.
\end{equation}
Then
\begin{equation}
\mathcal{H}^{f}(  \Lambda(\Psi)      )=  \mathcal{H}^{f}(  [0,1]^{mn}).   \end{equation}
\end{thm}

\begin{proof}
With $\delta=mn,$ and $\gamma=(m-1)n$ the function $g$ in \cite[Theorem 2]{BeDiVe_06}  becomes $g(r):=f(\Psi(r))\Psi(r)^{-\gamma}\rho(r)^{\gamma - \delta}=f(\Psi(r))\Psi(r)^{-(m-1)n}\rho(r)^{-n}.$ Also $\rho$ is 2-regular, thus from \cite[Theorem 2]{BeDiVe_06} it follows that
\begin{equation*}
\mathcal{H}^{f}(  \Lambda(\Psi)      )=\infty \quad \text{ if } \quad \sum^{\infty}_{n=1}g(2^{N})=\infty,
\end{equation*}
which is same as the  divergent sum condition  in \eqref{divcon}. 

Note that as the dimension function $r^{-nm}f(r) \to \infty$ as $r\to 0$ then $H^f (\Omega)=\infty$
and Theorem \ref{thm5} leads to the same conclusion as Theorem 2 in \cite{BeDiVe_06}.
%
 \end{proof}

\subsection{Dirichlet improvability and homogenous dynamics}
In one dimensional settings, continued fraction expansions have been useful in characterising $\psi$-Dirichlet improvable numbers \cite{KlWa_18}. However this machinery is not applicable in higher dimensions. For general dimensions, building on ideas from \cite{Da_85} (also see \cite{Kl_98}), a dynamical approach was proposed in \cite{KlWa_18}, reformulating the homogenous approximation problem as a shrinking target problem and a similar approach was used in \cite{KlWa_19} to solve an analogous inhomogeneous problem. Following the ideas from \cite{ KiKi_20, KlWa_19} we will use the
standard argument usually known as the `Dani correspondence' which serves as a connection between Diophantine approximation and homogenous dynamics. In order to describe how Dirichlet-improvability is related to dynamics we will start by recalling the dynamics on space of grids. To describe this dynamical interpretation, let us fix some notation. 

Fix $d=m+n.$ Let $$G_{d}=SL_{d}(\R) \text{ and }  \widehat{G}_{d}=ASL_{d}(\R)=G_{d} \rtimes \R^{d} $$ 
and put
$$\Gamma_{d}=SL_{d}(\Z) \text{ and }  \widehat{\Gamma}_{d}=ASL_{d}(\Z)=\Gamma_{d} \rtimes \Z^{d}.$$ 

Denote by $\widehat{Y}_{d}$ the space of affine shifts of  unimodular lattices in $\R^{d}$ (i.e. space of unimodular grids).
Clearly, $\widehat{Y}_{d}$ is canonically identified with $\widehat{G}_{d}/ \widehat{\Gamma}_{d}$ via 
$$<g,\ww>\widehat{\Gamma}_{d}\in \widehat{G}_{d}/ \widehat{\Gamma}_{d}  \longleftrightarrow g\Z^{d}+\ww\in \widehat{Y}_{d}$$ 
where
$<g,\ww>$ is an element of $\widehat{G}_{d}$ such that $g\in G_{d}$ and $\ww\in\R^{d}.$
Similarly, $Y_{d}:=G_{d}/{\Gamma}_{d}$ is identified with the space of unimodular lattices in $\R^{d}$ (i.e. the space of unimodular grids containing zero vector). Note that ${\Gamma}_{d}$ (respectively, $\widehat{\Gamma}_{d}$) is a lattice in $G_{d}$ (respectively, $\widehat{G}_{d}$). Denote by $m_{Y_{d}}$ the Haar probability measure on $Y_{d}.$
For any $t\in\R,$ the flow of interest $a_{t}$ is given by the diagonal matrix
$$ a_{t}:=\text{diag}(e^{t/m},\cdots,e^{t/m},e^{-t/n},\cdots, e^{-t/n}).$$
Let  
$$u_{A}:=\left( \begin{array}{cc} I_{m} & A \\
0 & I_{n} \end{array} \right)\in G_{d},$$
$$u_{A,\bb}:=\left<\left( \begin{array}{cc} I_{m} & A \\
0 & I_{n} \end{array} \right),
\left( \begin{array}{c} \bb \\ 0 \end{array} \right)\right>\in \widehat G_{d}$$
for $A\in X^{mn}$ and $(A,\bb)\in X^{mn}\times\R^{m}.$ 
Let us also denote by
$$\Lambda_{A}:=u_{A}\Z^{d}\in Y_{d} \text{ and }\Lambda_{A,\bb}:=u_{A,\bb}\Z^{d}\in \widehat{Y}_{d},$$ where 
$u_{A,\bb}\Z^{d}=\left\{\left( \begin{array}{cc} A\qq+\bb-\pp  \\
\qq  \end{array} \right):\pp\in \Z^{m},\qq\in \Z^{n}\right\}.$ 

Following \cite{KlWa_19}, define $\Delta:\widehat {Y}_{d}\to[-\infty,+\infty)$ by 
$$\Delta(\Lambda):=\log\inf_{\vv\in\Lambda}\|\vv\|.$$
\begin{lem}[\cite{KlMa_99}]\label{KlMa_19}
Let $\psi:[T_{0},\infty)\to\R_{+}$ be a continuous, non-increasing function where $T_{0}\in\R_{+}$ and $m,n$ be positive integers. Then there exists a continuous function
$$z=z_{\psi}:[t_{0},\infty)\to\R,$$
where $t_{0}:=\frac{m}{m+n}\log T_{0}-\frac{n}{m+n}\log\psi(T_{0}),$ such that
\begin{itemize}
\item [\rm{ (i) }] the function $t\mapsto t+nz(t)$ is strictly increasing and unbounded;
\item [\rm{ (ii) }]the function $t\mapsto t-mz(t)$ is non-decreasing;
\item [\rm{ (iii) }]$\psi(e^{t+nz(t)})=e^{-t+mz(t)}$ for all $t\geq t_{0}.$
\end{itemize}
\end{lem}
Note that, properties $(i)$ and $(ii)$ of Lemma \ref{KlMa_19} imply that any $z=z_{\psi}$ does not oscillate too wildly. Namely,
$ z(s)-\frac{1}{m}\leq z(u)\leq z(s)+\frac{1}{n}  \text{ whenever } s\leq u\leq s+1. $

The following lemma, which rephrases $\psi$-Dirichlet improvable properties of $(A,\bb)\in  X^{mn}\times\R^{m}$ as the statement about the orbit of $\Lambda_{A,\bb}$ in the dynamical space $(\widehat{Y}_{d},a_{t})$, is the general version of the correspondence between the improvability of the inhomogeneous Dirichlet theorem and dynamics on $\widehat{Y}_{d}.$
\begin{lem}[\cite{KlWa_19}]\label{KKL2.3}
Let $z=z_{\psi}$ be the function associated to $\psi$ by Lemma \ref{KlMa_19}. Then $(A,\bb)\in\widehat{D}_{m,n}(\psi)$ if and only if $\Delta(a_{t}\Lambda_{A,\bb})<z_{\psi}(t)$ for all sufficiently large $t.$
\end{lem}

This equivalence is usually called the Dani Correspondence. In view of this interpretation a pair fails to be $\psi$-Dirichlet improvable if and only if the associated grid visits the target $\Delta^{-1}([z_{\psi}(t),\infty))$ at unbounded times $t$ under the flow $a_{t}.$ Note that from the above lemma in the definitions  $\widehat{D}_{m,n}(\psi)^{c}=\limsup\limits_{t\to\infty} \left\{ (A,\bb):\Delta(a_{t}\Lambda_{A,\bb})\geq z_{\psi}(t)\right\}  $ and $\widehat{D}^{\bb}_{m,n}(\psi)^{c}=\limsup\limits_ {t\to\infty}\left\{ A:\Delta(a_{t}\Lambda_{A,\bb})\geq z_{\psi}(t)\right\},$ the limsup is taken for real values $t\in\R$.
 However to prove the convergent part, we need to use Hausdorff--Cantelli lemma (Lemma \ref{CB}), therefore we will consider limsup sets taken for $t\in\N.$ Thus we will use the following definitions: there exists a non-zero positive constant $C_{0}$ such that
\begin{align}\label{alDI}
\widehat{D}_{m,n}(\psi)^{c}&\subseteq \limsup\limits_{t\to\infty,t\in\N}\left\{ (A,\bb):\Delta(a_{t}\Lambda_{A,\bb})\geq z_{\psi}(t)-C_{0}\right\},\\  \label{alDI1}\widehat{D}^{\bb}_{m, n}(\psi)^{c}&\subseteq\limsup\limits_ {t\to\infty,t\in\N}\left\{ A:\Delta(a_{t}\Lambda_{A,\bb})\geq z_{\psi}(t)-C_{0}\right\}. \end{align}
The validity of these definitions can be observed by the fact that $z_{\psi}$ does not oscillate wildly by \cite[Remark 3.3]{KlWa_19} and $\Delta$ is uniformly  continuous on the set $\Delta^{-1}([z,\infty))$ for any $z\in\R,$
(\cite[Lemma 2.1]{KlWa_19}).

\section{Proof of Theorem \ref{BST1} and \ref{BST2}: the convergent case}

\begin{lem}\label{BS1}
Let $\psi:[T_{0},\infty)\rightarrow {\R_{+}}$ be a non-increasing function, and let $z=z_{\psi}$ be the function associated to $\psi$ by Lemma \ref{KlMa_19}. Let $f $ be a dimension function satisfying \eqref{C1} and \eqref{C2} where $nm-n<s\leq nm.$ Also suppose that \eqref{C05} holds.
Then we have
\begin{equation*}
\sum\limits_{q=\lceil{T_{0}}\rceil}^{\infty} \frac{1}{\psi(q)q^{2}} \left( \frac{ q^{\frac{1}{n}}}{\psi(q)^\frac{1}{m}}\right)^{mn} f\left( \frac{\psi(q)^\frac{1}{m}}{ q^{\frac{1}{n}}}\right)<\infty 
\iff  \sum\limits_{t=\lceil{t_{0}}\rceil}^{\infty}e^{-(m+n)z(t)}e^{(m+n)t}f(e^{-\frac{(m+n)t}{mn}})<\infty.
\end{equation*}
\end{lem}
\begin{proof}
The proof of this lemma uses ideas introduced in \cite[Lemma 8.3]{KlMa_99} and \cite{KlWa_19}. Using the monotonicity of $\psi$ and \cite[Remark 3.3]{KlWa_19}, let us replace the sums with integrals
\begin{equation*}
\int_{T_{0}}^{\infty}     \frac{1}{\psi(x)x^{2}} \left( \frac{ x^{\frac{1}{n}}}{\psi(x)^\frac{1}{m}}\right)^{mn} f\left( \frac{\psi(x)^\frac{1}{m}}{ x^{\frac{1}{n}}}\right)dx
\text{ and }
\int_{t_{0}}^{\infty}e^{-(m+n)z(t)}e^{(m+n)t}f(e^{-\frac{(m+n)t}{mn}})dt.
\end{equation*}
Define
\begin{equation*}
P:=-\log\circ\psi\circ\exp:[T_{0},\infty)\rightarrow\R \text{ and } \lambda(t):=t+nz(t).
\end{equation*}
Since $\psi(e^{\lambda})=e^{-P(\lambda)}$, letting $\log x=\lambda$ we have
\begin{align}
\int_{T_{0}}^{\infty}     \frac{1}{\psi(x)x^{2}} \left( \frac{ x^{\frac{1}{n}}}{\psi(x)^\frac{1}{m}}\right)^{mn} f\left( \frac{\psi(x)^\frac{1}{m}}{ x^{\frac{1}{n}}}\right)dx&=&
\int_{\log T_{0}}^{\infty}     \frac{1}{\psi(e^{\lambda})e^{2\lambda}} \left( \frac{ e^{\lambda m}}{\psi(e^{\lambda})^{n}}\right) f\left( \frac{\psi(e^{\lambda})^\frac{1}{m}}{ e^{\frac{\lambda}{n}}}\right)e^{\lambda}d\lambda \notag \\
&=&  \int_{\log T_{0}}^{\infty}  e^{(m-1)\lambda}e^{(1+n)P(\lambda)}f(e^{\frac{-P(\lambda)}{m} }   e^{\frac{-\lambda}{n}}               )             d\lambda. \label{Ea}
\end{align}

Using $P(\lambda(t))=t-mz(t),$ we have
\begin{align}
\int_{t_{0}}^{\infty}e^{-(m+n)z(t)}e^{(m+n)t}f(e^{-\frac{(m+n)t}{mn}})dt&= \int_{t_{0}}^{\infty} e^{(m-1)\lambda}e^{(1+n)P(\lambda)}f(e^{\frac{-P(\lambda)}{m} }   e^{\frac{-\lambda}{n}}               )             d\left[\frac{m}{m+n}\lambda+\frac{n}{m+n}P(\lambda)\right] 
\notag\\
&=\frac{m}{m+n} \int_{\log T_{0}}^{\infty}  e^{(m-1)\lambda}e^{(1+n)P(\lambda)}f(e^{\frac{-P(\lambda)}{m} }   e^{\frac{-\lambda}{n}}  )  d\lambda
\notag\\ 
&+ \frac{n}{m+n}\int_{\log T_{0}}^{\infty}  e^{(m-1)\lambda}e^{(1+n)P(\lambda)}f(e^{\frac{-P(\lambda)}{m} }   e^{\frac{-\lambda}{n}}               ) d(P(\lambda)) .  \label{sat}          \end{align}

The  term in the last line can be expressed by 
\begin{align}
&\frac{n}{m+n}\int_{\log T_{0}}^{\infty}  e^{(m-1)\lambda}e^{(1+n)P(\lambda)}f(e^{\frac{-P(\lambda)}{m} }  e^{\frac{-\lambda}{n}} ) d(P(\lambda))
\notag \\
&\asymp\frac{n}{m+n}   \int_{\log T_{0}}^{\infty}  e^{(m-1)\lambda}f(   e^{\frac{-\lambda}{n}}  )e^{(1+n)P(\lambda)}e^{-s\frac{P(\lambda)}{m} }    d(P(\lambda)),
\notag \\
&=\frac{n}{m+n}\Big{(1+\frac{mn-s}{m}}\Big)^{-1} \int_{\log T_{0}}^{\infty}  e^{(m-1)\lambda}f(   e^{\frac{-\lambda}{n}}  )  d(e^{(  (1+n) -\frac{s}{m}      )P(\lambda)}), \label{LE}
\end{align}
the second last equation follows from \eqref{C1} and \eqref{C05}. 
Since by using \eqref{C05} and the fact that $\psi(e^{\lambda})=e^{-P(\lambda)}$ we obtain the condition
$$(e^{\frac{-\lambda}{n}})^{\alpha}\leq e^{\frac{-P(\lambda)}{m} } \leq (e^{\frac{-\lambda}{n}})^{\frac{1}{\alpha}},$$
therefore by using \eqref{C1} we can write
$$f(e^{\frac{-P(\lambda)}{m} }   e^{\frac{-\lambda}{n}} )\asymp  e^{-s\frac{P(\lambda)}{m} } f(   e^{\frac{-\lambda}{n}}).$$

Next we will use integration by parts to evaluate the integral in \eqref{LE}.
\begin{align}
&\int_{\log T_{0}}^{\infty}  e^{(m-1)\lambda}f(   e^{\frac{-\lambda}{n}}  )  d(e^{(  (1+n) -\frac{s}{m}      )P(\lambda)})\notag\\
&=- \int_{\log T_{0}}^{\infty} [(m-1) e^{(m-1)\lambda}f(e^{-\frac{\lambda}{n}}) - \frac 1n e^{(m-1)\lambda} e^{\frac{-\lambda}n} f'(e^{\frac{-\lambda}n}) ] e^{((1+n)-\frac{s}{m})P(\lambda)}d\lambda\notag\\
&+e^{(m-1)\lambda}f(e^{-\frac{\lambda}{n}}) e^{((1+n)-\frac{s}{m})P(\lambda)}\Big|^{\infty}_{\log T_{0}}\notag\\
&= \int_{\log T_{0}}^{\infty}\left( (1-m) e^{(m-1)\lambda} f(e^{-\frac{\lambda}{n}})e^{((1+n)-\frac{s}{m})P(\lambda)} +
\frac{1}{n} e^{(m-1)\lambda}f'(e^{-\frac{\lambda}{n}})e^{-\frac{\lambda}{n} }   e^{((1+n)-\frac{s}{m})P(\lambda)}\right)d\lambda \notag
 \\
&+\lim_{\lambda\to\infty} e^{(m-1)\lambda} f(e^{-\frac{\lambda}{n}})e^{((1+n)-\frac{s}{m})P(\lambda)}-T^{m-1}_{0}f(T^{-\frac{1}{n}}_{0})\psi(T_{0})^{-((1+n)-\frac{s}{m})},\notag \\
& \text{ by } \eqref{C2}, \text{ we have } f'(e^{-\frac{\lambda}{n}})= a(e^{-\frac{\lambda}{n}})\frac{f(e^{-\frac{\lambda}{n}})} {e^{-\frac{\lambda}{n}}},\notag\\
&= \int_{\log T_{0}}^{\infty} \left(1-m + \frac 1n a(e^{-\frac \lambda n}) \right)e^{(m-1)\lambda} f(e^{-\frac{\lambda}{n}})e^{((1+n)-\frac{s}{m})P(\lambda)} d\lambda
\notag \\
&+\lim_{\lambda\to\infty} e^{(m-1)\lambda} f(e^{-\frac{\lambda}{n}})e^{((1+n)-\frac{s}{m})P(\lambda)}-T^{m-1}_{0}f(T^{-\frac{1}{n}}_{0})\psi(T_{0})^{-((1+n)-\frac{s}{m})}\notag \\
&\asymp\int_{\log T_{0}}^{\infty} (1-m + \frac 1n a(e^{-\frac \lambda n}))e^{(m-1)\lambda} e^{(1+n)P(\lambda)}f(e^{\frac{-P(\lambda)}{m} }   e^{\frac{-\lambda}{n}}               )  d\lambda
 \\
&+\lim_{\lambda\to\infty} e^{(m-1)\lambda} e^{(1+n)P(\lambda)}f(e^{\frac{-P(\lambda)}{m}    e^{\frac{-\lambda}{n}}               )} -T^{m-1}_{0}f(T^{-\frac{1}{n}}_{0})\psi(T_{0})^{-((1+n)-\frac{s}{m})}.\label{Fri}
\end{align}
Note that as $\lambda\to\infty,$ $e^{-\frac{\lambda}{n}}\to0$ thus by assumption $a(e^{-\frac{\lambda}{n}})\to s$ and therefore 
$$\left(1-\frac{mn-a(e^{-\frac{\lambda}{n}})}{n}\right)\to \left(1-\frac{mn-s}{n}\right),$$ which is finite and positive for $T_{0}$ large enough (since $s > nm - n$). 
Observe that
$$\lim_{\lambda\to\infty} e^{(m-1)\lambda}e^{(1+n)P(\lambda)}f(e^{\frac{-P(\lambda)}{m} }   e^{\frac{-\lambda}{n}}) =0$$ if the integral 
$$\int_{\log T_{0}}^{\infty}  e^{(m-1)\lambda}e^{(1+n)P(\lambda)}f(e^{\frac{-P(\lambda)}{m} }   e^{\frac{-\lambda}{n}} )             d\lambda$$ converges. 
Thus the convergence of 
$$\int_{T_{0}}^{\infty}     \frac{1}{\psi(x)x^{2}} \left( \frac{ x^{\frac{1}{n}}}{\psi(x)^\frac{1}{m}}\right)^{mn} f\left( \frac{\psi(x)^\frac{1}{m}}{ x^{\frac{1}{n}}}\right)dx \text{ or } \int_{t_{0}}^{\infty}e^{-(m+n)z(t)}e^{(m+n)t}f(e^{-\frac{(m+n)t}{mn}})dt 
$$
 implies the convergence of other since all summands are positive  except the finite value
\linebreak{$-T^{m-1}_{0}f(T^{-\frac{1}{n}}_{0})\psi(T_{0})^{-((1+n)-\frac{s}{m})}.$}
\end{proof}

In order to apply the Hausdorff--Cantelli lemma (Lemma \ref{CB}) we need a sequence of coverings for the sets $\widehat{D}_{m,n}(\psi)^{c}$ and $\widehat{D}^{\bb}_{m,n}(\psi)^{c}.$ Recall that we are considering the supremum norm $\|\cdot\|$ on $[0,1]^{mn}$ and let $\lambda_{j}(\Lambda)$ denote the $j$-th successive minimum of a lattice $\Lambda\subseteq\R^{d}$ i.e. the infimum of $\lambda$ such that the ball $B^{\R^{d}}_{\lambda}(0)$ contains $j$ independent vectors of $\Lambda.$ Then:
\begin{pro}[Kim--Kim, {\cite[Proposition 3.6]{KiKi_20}}]\label{KKP3.6}
Let $C_{0}$ be the same constant as in \eqref{alDI} and \eqref{alDI1}. For $t\in\N,$ let $Z_{t}:=\{   A\in[0,1]^{mn}:\log(d\lambda_{d}(a_{t}\Lambda_{A}))\geq z_{\psi}(t)-C_{0} \}.$ Then $Z_{t}$ can be covered with $Ke^{(m+n)(t-z_{\psi}(t))}$ balls in $X^{mn}=M_{m,n}(\R)$ of radius 
$\frac{1}{2}e^{-(\frac{1}{m}+\frac{1}{n})t}$ for a constant $K>0$ not depending on $t.$
\end{pro}

We are now in a position to prove the following statement.
\begin{pro}\label{BS3.8}
Let $mn-n< s\leq mn.$ If  $\sum\limits_{q=1}^{\infty} \frac{1}{\psi(q)q^{2}} \left( \frac{ q^{\frac{1}{n}}}{\psi(q)^\frac{1}{m}}\right)^{mn} f\left( \frac{\psi(q)^\frac{1}{m}}{ q^{\frac{1}{n}}}\right)<\infty, $ then $\mathcal{H}^{f}( \limsup\limits_{t\to\infty}  Z_{t}    )=0$
and 
$\mathcal{H}^{f+m}( \limsup\limits_{t\to\infty}  Z_{t} \times [0,1]^{m})=0.$ $($Note that $\mathcal H^{f+m}$ represents the Hausdorff measure of a set when we take $(f+m)(r)=r^{m}f(r)).$
\end{pro}
\begin{proof}
By Lemma \ref{BS1}, the assumption $\sum\limits_{q=1}^{\infty} \frac{1}{\psi(q)q^{2}} \left( \frac{ q^{\frac{1}{n}}}{\psi(q)^\frac{1}{m}}\right)^{mn} f\left( \frac{\psi(q)^\frac{1}{m}}{ q^{\frac{1}{n}}}\right)<\infty$  is equivalent to 
\begin{align}
& \sum\limits_{t=1}^{\infty}e^{-(m+n)(z(t)-t)}f(e^{-(\frac{1}{m}+\frac{1}{n})t})<\infty.
\end{align}
For each $t\in\N,$ let $D_{t,{1}}, D_{t,{2}},\cdots,D_{t,{p_{t}}}$ be the balls of radius $\frac{1}{2}e^{-(\frac{1}{m}+\frac{1}{n})t}$ covering $Z_{t}$ as in Proposition \ref{KKP3.6}. Note that $p_{t},$ the number of the balls, is not greater than $Ke^{(m+n)(t-z_{\psi}(t))}$ by Proposition \ref{KKP3.6}. By applying Lemma \ref{CB} to the sequence of balls $\{D_{t,{j}}\}_{t\in\N,1\leq j \leq p_{t}},$ we have $\mathcal{H}^{f}( \limsup\limits_{t\to\infty}  Z_{t}    )\leq\mathcal{H}^{f}( \limsup\limits_{t\to\infty}  D_{t_{j}}    )=0.$

We prove the second statement by a similar argument. Proposition \ref{KKP3.6} implies that $Z_{t}\times[0,1]^{m}$ can be covered with $Ke^{\frac{m+n}{n}t}e^{(m+n)(t-z_{\psi}(t))}$ balls of radius $\frac{1}{2}e^{-(\frac{1}{m}+\frac{1}{n})t}.$ Applying Lemma \ref{CB} again, we have $\mathcal{H}^{f
+m}( \limsup\limits_{t\to\infty}  Z_{t} \times [0,1]^{m})=0.$
\end{proof}

The convergence parts of Theorems \ref{BST1} and \ref{BST2} follow from this proposition. We will adapt a similar method as in \cite{KiKi_20}.
\begin{proof}
We first prove the singly metric case i.e., the convergent part of Theorem \ref{BST2}. We claim that $\log
(d\lambda_{d}(a_{t}\Lambda_{A}))\geq\Delta(a_{t}\Lambda_{A,\bb})$ for every $\bb\in\R^{m}.$  Let $v_{1},\dots,v_{d}$ be linearly independent vectors satisfying $\|v_{i}\|\leq\Lambda_{d}(a_{t}\Lambda_{A})$ for $1\leq i \leq d.$ The shortest vector of $a_{t}\Lambda_{A,\bb}$ can be written as a form of $\sum_{1}^{d}\alpha_{i}v_{i}$ for some $-1\leq\alpha_{i}\le1$, so the length of the shortest vector is less than $\sum_{1}^{d}\|v_{i}\|.$ Thus, $ \Delta(a_{t}\Lambda_{A,\bb})\leq\log\sum\limits_{i}^{d}\|v_{i}\|\leq \log(d\lambda_{d}(a_{t}\Lambda_{A})).$ This implies  $\widehat{D}^{\bb}_{m, n}(\psi)^{c}\subseteq\limsup\limits_{t\to\infty}  \{ A\in[0,1]^{mn}:\Delta(a_{t}\Lambda_{A,\bb})\geq z_{\psi}(t)-C_{0}\}\subseteq \limsup\limits_{t\to\infty}  Z_{t} $ by Lemma \ref{KKL2.3} and Proposition \ref{BS3.8}, thus we obtain $\mathcal{H}^{f}(  \hat{D}^{\bb}_{m, n}(\psi)^{c} )\leq\mathcal{H}^{f}( \limsup\limits_{t\to\infty}  Z_{t}    )=0$.

Similarly for the doubly metric case, together with the second statement of Proposition \ref{BS3.8}, $\widehat{D}_{m, n}(\psi)^{c} \subseteq \limsup\limits_{t\to\infty} \{   (A,\bb)\in[0,1]^{mn+m}:   \Delta(a_{t}\Lambda_{A,\bb})\geq z_{\psi}(t)-C_{0}\}\subseteq \limsup\limits_{t\to\infty}Z_{t}\times[0,1]^{m}$ provides the proof of the convergent part of Theorem \ref{BST1}.
 \end{proof}

\section{Proof of Theorem \ref{BST1} and \ref{BST2}: the divergent case}\label{Sec4}
Recall that $d=m+n$ and assume that $\psi:[T_{0},\infty)\to\R_{+}$ is a decreasing function satisfying $\lim_{T\to\infty}\psi(T)=0.$ Denote by $\|\cdot\|_{\Z}$ and $|.|_{\Z}$ the distance to the nearest integer vector and number, respectively. Define the function $\widetilde{\psi}:[S_{0},\infty)\to\R_{+}$ by 
$$\widetilde{\psi}(S)=(\psi^{-1}(S^{-m})))^{\frac{-1}{n}}$$
where $S_{0}=\psi(T_{0})^{\frac{-1}{m}}.$ The next lemma associates $\psi$-Dirichlet non-improvability with $\widetilde{\psi}$-approximability via a transference lemma as follows. 
\begin{lem}\cite[Lemma 4.2]{KiKi_20}\label{Transf}
Given $(A,\bb)\in X^{mn}\times\R^{m},$ if the system 
\begin{equation*}
\|A^{t}\xx\|_{\Z}<d^{-1}|\bb\cdot \xx|_{\Z}\widetilde{\psi}(S) \text{ and } \|\xx\|<d^{-1}|\bb\cdot \xx|_{\Z}S
\end{equation*}
has a nontrivial solution $\xx\in\Z^{m}$ for an unbounded set of $S\geq S_{0},$ then $(A,\bb)\in \widehat{D}_{m,n}(\psi)^{c}.$
\end{lem}

Following \cite{KiKi_20} we adopt some notations.
Let $W_{S,\varepsilon}$ be the set of $A\in[0,1]^{mn}$ such that there exists
$\xx_{A,S}\in\Z^{m}\setminus \{\0\}$ satisfying 
$$\|A^{t}\xx_{A,S}\|_{\Z}<d^{-1}\varepsilon \widetilde{\psi}(S) \text{ and } \|\xx_{A,S}\|<d^{-1}\varepsilon S$$
and let
$$\widehat{W}_{S,\varepsilon}:=\{ (A,\bb)\in [0,1]^{mn+m}: A\in W_{S,\varepsilon }\text{ and }  |\bb\cdot\xx_{A,S}|_{\Z}>\varepsilon \}.$$

For fixed $\bb\in\R^{m},$ consider the set $W_{\bb,S,\varepsilon}$ of matrices $A\in[0,1]^{mn}$
such that there exists $\xx\in\Z^{m}\setminus\{\0\}$ satisfying 
\begin{itemize}
\item  $|\bb\cdot\xx|_{\Z}>\varepsilon$
\item $\|A^{t}\xx\|_{\Z}<d^{-1}\varepsilon\widetilde{\psi}(S)$ and $\|\xx\|<d^{-1}\varepsilon S.$
\end{itemize}
 Let $W_{\bb,\varepsilon}:=\limsup\limits_{S\to\infty}W_{\bb,S,\varepsilon}.$ Note that $A\in W_{S,\varepsilon}$ if and only if
$$\|A^{t}\xx_{A,S}\|_{\Z}<\Psi_{\varepsilon}(U) \text{ and } \|\xx_{A,S}\|<U \text{ for some } \xx_{A,S},$$
where 
\begin{equation}
\Psi_{\varepsilon}(U):=d^{-1}\varepsilon\widetilde{\psi}(d\varepsilon^{-1}U), \;\;\; U = d^{-1} \varepsilon S.
\end{equation}
By Lemma \ref{Transf} $\limsup\limits_{S\to\infty}\widehat{W}_{S,\varepsilon}\subseteq \widehat{D}_{m,n}(\psi)^{c}$ and $W_{\bb,\varepsilon}\subseteq \widehat{D}^{\bb}_{m,n}(\psi)^{c}.$

Further $\limsup\limits_{S\to\infty}{W}_{S,\varepsilon}=\{A\in[0,1]^{mn}: A^{t}\in W_{n,m}(\Psi_{\varepsilon})\}$ is the set of matrices whose transposes are $\Psi_{\varepsilon}$-approximable. From here onwards we use a slightly different definition of $\Psi_{\varepsilon}$-approximability; recall from footnote \ref{FT1} where the inequality $\|A^{t}\xx\|_{\Z}<\Psi_{\varepsilon}(\|\xx\|)$ is used instead of \eqref{AS2.1}. Then, $W_{\bb,\varepsilon}$ can be considered as the set of matrices whose transposes are $\Psi_{\varepsilon}$-approximable with solutions restricted on the set $\{\xx\in\Z^{m}:|\bb\cdot\xx|_{\Z}>\varepsilon\}.$

\subsection{Mass distributions on $\Psi_{\varepsilon}$-approximable matrices}
In this subsection we prove the divergent part of Theorem \ref{BST1} using mass distributions on $\Psi_{\varepsilon}$-approximable matrices following \cite{AlBe_18}.
\begin{lem}\label{BS4.2}
For each $mn-n< s\leq mn$ and $0<\varepsilon< 1/2,$ let $U_{0}=d^{-1}\varepsilon S_{0}$ and f be a dimension function satisfying \eqref{C1} and \eqref{C2}. Suppose that \eqref{C05} holds.
 Then
\begin{equation*}
\sum\limits_{q=\lceil{T_{0}}\rceil}^{\infty} \frac{1}{\psi(q)q^{2}} \left( \frac{ q^{\frac{1}{n}}}{\psi(q)^\frac{1}{m}}\right)^{mn} f\left( \frac{\psi(q)^\frac{1}{m}}{ q^{\frac{1}{n}}}\right)<\infty 
\iff  \sum\limits_{h=\lceil{U_{0}}\rceil}^{\infty}h^{m+n-1}\Big(\frac{\Psi_{\varepsilon}(h)}{h}     \Big)^{-n(m-1)}f \Big(\frac{\Psi_{\varepsilon}(h)}{h}     \Big)  <\infty.
\end{equation*}
\end{lem}
\begin{proof}
Similar to Lemma \ref{BS1}, we may replace the sums with integrals
\begin{equation*}
\int^{\infty}_{T_{0}}\frac{1}{\psi(x)x^{2}} \left( \frac{ x^{\frac{1}{n}}}{\psi(x)^\frac{1}{m}}\right)^{mn} f\left( \frac{\psi(x)^\frac{1}{m}}{ x^{\frac{1}{n}}}\right)dx \text{ and } 
\int\limits_{U_0}^{\infty}h^{m+n-1}\Big(\frac{\Psi_{\varepsilon}(h)}{h}     \Big)^{-n(m-1)}f \Big(\frac{\Psi_{\varepsilon}(h)}{h}     \Big)dh,
 \end{equation*}
respectively.

Note that since $\Psi_{\varepsilon}(h)=d^{-1}\varepsilon\widetilde\psi(d\varepsilon^{-1}h),$ if we consider the term $\int\limits_{U_0}^{\infty}h^{m+n-1}\Big(\frac{\Psi_{\varepsilon}(h)}{h}     \Big)^{-n(m-1)}f \Big(\frac{\Psi_{\varepsilon}(h)}{h}     \Big)dh,$ then
\begin{equation*}
\int\limits_{U_0}^{\infty}h^{m+n-1}\Big(\frac{\Psi_{\varepsilon}(h)}{h}     \Big)^{-n(m-1)}f \Big(\frac{\Psi_{\varepsilon}(h)}{h}     \Big) dh  <\infty
\iff  \int\limits_{S_0}^{\infty}y^{m+n-1}\Big(\frac{\widetilde\psi(q)}{y}     \Big)^{-n(m-1)}f \Big(\frac{\widetilde\psi(q)}{y}     \Big) dy  <\infty.
\end{equation*}
Also, since $\widetilde\psi(y)=\psi^{-1}(y^{-m})^{-\frac{1}{n}},$ we have
\begin{align}
\int^{\infty}_{S_{0}}y^{m+n-1}  \left(\frac{\widetilde\psi(y)}{y}     \right)^{-n(m-1)}f \left(\frac{\widetilde\psi(y)}{y}     \right)dy&=\int^{\infty}_{S_{0}} y^{mn+m-1}(\psi^{-1}(y^{-m}))^{m-1} f\left( \frac{(\psi^{-1}(y^{-m}))^{-\frac{1}{n}}}{y}\right)dy \notag
\\&=\frac{1}{m}\int_{S^{m}_{0}}^{\infty}t^{n} (\psi^{-1}(t^{-1}))^{m-1}f\left( \frac{(\psi^{-1}(t^{-1}))^{-\frac{1}{n}}}{t^{\frac{1}{m}}}\right)dt \notag \\
&= \frac{1}{m}\int_{\psi^{-1}(S^{-m}_{0})}^{\infty}x^{m-1} (\psi(x)^{-1})^{n}f\left( \frac{x^{-\frac{1}{n}}}{(\psi(x)^{-1})^{\frac{1}{m}}}\right)d\psi(x)^{-1} \notag
\\
&\asymp \frac{1}{m}\left(   n-\frac{s}{m} +1\right)^{-1}\int^{\infty}_{T_{0}}x^{m-1}f(x^{-\frac{1}{n}}) d(\psi(x)^{-1})^{n-\frac{s}{m}+1},\notag
 \end{align}
where in the second last line we used the change of variables $x= \psi^{-1}(t^{-1}),$ $t=\psi(x)^{-1}$ and in the last line we used \eqref{C1} and \eqref{C05}.
Since it follows from \eqref{C05} that
$(x^{-\frac{1}{n}})^{\alpha}\leq(\psi(x)^{-1})^{-\frac{1}{m}} \leq (x^{-\frac{1}{n}})^\frac{1}{\alpha}.$
Therefore by using \eqref{C1} we can write 
$$  f((\psi(x)^{-1})^{-\frac{1}{m}}x^{-\frac{1}{n}})  \asymp (\psi(x)^{-1})^{-\frac{s}{m}}f(x^{-\frac{1}{n}}). $$

Using integration by parts 
\begin{align}\label{sumT}
&\int^{\infty}_{T_{0}}x^{m-1}f(x^{-\frac{1}{n}}) d(\psi(x)^{-1})^{n-\frac{s}{m}+1}\notag\\
&=(\lim_{x\to\infty}x^{m-1}\psi(x)^{-n-1+\frac{s}{m}}f(x^{-\frac{1}{n}})-T^{m-1}_{0}\psi(T_{0})^{-n-1+\frac{s}{m}}f(T_{0}^\frac{-1}{n}) )\notag\\
&+ \int_{T_{0}}^{\infty}\left[-(m-1)x^{m-2}f(x^{-\frac{1}{n}})+\frac{1}{n}  x^{m-1}x^{-\frac{1}{n}-1}f'(x^{-\frac{1}{n}})\right]\psi(x)^{-n-1+\frac{s}{m}}dx\notag\\
&=\lim_{x\to\infty}x^{m-1}\psi(x)^{-n-1+\frac{s}{m}}f(x^{-\frac{1}{n}})-T^{m-1}_{0}\psi(T_{0})^{-n-1+\frac{s}{m}}f(T_{0}^\frac{-1}{n}) \notag\\
&+\int_{T_{0}}^{\infty}\left[-(m-1)+\frac{1}{n}a(x^{-\frac{1}{n}})\right]x^{m-2}f(x^{-\frac{1}{n}})\psi(x)^{-n-1+\frac{s}{m}}dx, \text{ by \eqref{C2}} \notag\\
&\asymp \lim_{x\to\infty}x^{m-1}\psi(x)^{-n-1} f\left(\frac{\psi(x)^{\frac{1}{m}}}{x^{\frac{1}{n}}}\right) -T^{m-1}_{0}\psi(T_{0})^{-n-1+\frac{s}{m}}f\left(\frac{\psi(T_{0})^{\frac{1}{m}}}{T_{0}^{\frac{1}{n}}}\right) \notag\\
&+
\int_{T_{0}}^{\infty}\left[\frac{1}{n}a(x^{-\frac{1}{n}}) - (m-1)\right]x^{m-2}\psi(x)^{-n-1}f\left(\frac{\psi(x)^{\frac{1}{m}}}{x^{\frac{1}{n}}}\right)dx.  \notag
\end{align}
Note that 
\begin{align}
&\int_{T_{0}}^{\infty}x^{m-2}\psi(x)^{-n-1}f\left(\frac{\psi(x)^{\frac{1}{m}}}{x^{\frac{1}{n}}}\right)dx
=\int_{T_{0}}^{\infty}x^{m-1}\psi(x)^{-n-1}f\left(\frac{\psi(x)^{\frac{1}{m}}}{x^{\frac{1}{n}}}\right)d\log x.
\end{align}

Thus the convergence of $\int^{\infty}_{T_{0}}x^{m-2}\psi(x)^{-1-n}f\left(\frac{\psi(x)^{\frac{1}{m}}}{x^{\frac{1}{n}}}\right)dx$ gives that 
\begin{equation*}\lim_{x\to\infty}x^{m-1}\psi(x)^{-n-1} f\left(\frac{\psi(x)^{\frac{1}{m}}}{x^{\frac{1}{n}}}\right)<\infty. 
\end{equation*}

Also observe that as $x\to\infty,$ $a(x^\frac{-1}{n})\to s.$ Therefore $$\frac{1}{n}a(x^\frac{-1}{n})-(m-1)\to \frac{s-n(m-1)}{n}$$
which is finite and positive (since $s>mn-n$). Therefore the convergence of $\int^{\infty}_{T_{0}}x^{m-2}\psi(x)^{-1-n}f\left(\frac{\psi(x)^{\frac{1}{m}}}{x^{\frac{1}{n}}}\right)dx$ gives the convergence of 

\begin{equation*}\frac{1}{n}
\int_{T_{0}}^{\infty}a(x^{-\frac{1}{n}})x^{m-2}\psi(x)^{-n-1}f\left(\frac{\psi(x)^{\frac{1}{m}}}{x^{\frac{1}{n}}}\right)dx   -(m-1)\int_{T_{0}}^{\infty}x^{m-2}\psi(x)^{-n-1} f\left(\frac{\psi(x)^{\frac{1}{m}}}{x^{\frac{1}{n}}}\right) dx. 
\end{equation*}
Hence the convergence of 
\begin{equation*}
\int^{\infty}_{T_{0}}\frac{1}{\psi(x)x^{2}} \left( \frac{ x^{\frac{1}{n}}}{\psi(x)^\frac{1}{m}}\right)^{mn} f\left( \frac{\psi(x)^\frac{1}{m}}{ x^{\frac{1}{n}}}\right)dx \text{ or } \int^{\infty}_{S_{0}} y^{m+n-1}  \left(\frac{\widetilde\psi(y)}{y}     \right)^{-n(m-1)}f \left(\frac{\widetilde\psi(y)}{y}     \right)dy,     \end{equation*}
implies the convergence of other one since for $T_{0}$ large enough all summands in \eqref{sumT} are positive except the finite value
\[
-T^{m-1}_{0}\psi(T_{0})^{-n-1+\frac{s}{m}}f\left(\frac{\psi(T_{0})^{\frac{1}{m}}}{T_{0}^{\frac{1}{n}}}\right).
\]
\end{proof}

\begin{lem}[{\cite[Section 5]{AlBe_18}}]\label{AlBe}
Assume that 
$\sum\limits_{q=1}^{\infty} \frac{1}{\psi(q)q^{2}} \left( \frac{ q^{\frac{1}{n}}}{\psi(q)^\frac{1}{m}}\right)^{mn} f\left( \frac{\psi(q)^\frac{1}{m}}{ q^{\frac{1}{n}}}\right)=\infty.$ Fix $0<\varepsilon<\frac{1}{2}.$ Then, for any $\eta>1$ there exists a probability measure $\mu$ on $\limsup W_{S,\varepsilon}$ satisfying the condition that for an arbitrary ball $D$ of sufficiently small radius $r(D)$ we have 
$$\mu(D)\ll\frac{f(r(D))}{\eta},$$
where the implied constant does not depend on $D$ or $\eta.$
\end{lem}


\begin{proof}
Note that $\limsup_{S\to\infty}W_{S,\varepsilon}=\{ A\in [0,1]^{mn}: A^{t} \in W_{n,m}  (\Psi_{\varepsilon})   \}.$ By Lemma \ref{BS4.2}

$$\sum\limits_{h=1}^{\infty}h^{m+n-1}\Big(\frac{\Psi_{\varepsilon}(h)}{h}     \Big)^{-n(m-1)}f \Big(\frac{\Psi_{\varepsilon}(h)}{h}     \Big) =\infty,$$
which is the divergent assumption of Theorem \ref{AlBeJT} for $W_{n,m}  (\Psi_{\varepsilon}).$ From the proof of Jarnik's Theorem in \cite{AlBe_18} and the construction of probability measure in \cite[Section 5]{AlBe_18} we can obtain a probability measure $\mu$ on $\limsup_{S\to\infty}W_{S,\varepsilon}$ satisfying the above condition.
\end{proof}
Let us prove the divergent part of Theorem \ref{BST1}.
\begin{proof}
Assume that $mn+m-n<s<mn+m$ and fix $0<\varepsilon<\frac{1}{2}.$ For any fixed $\eta>1,$ let $\mu$ be a probability measure on $\limsup_{S\to\infty}W_{S,\varepsilon}$ as in Lemma \ref{AlBe} with $f(r(D))$ replaced by $r(D)^{-m}f(r(D)).$ 

Here we remark that since $f(r)$ satisfies \eqref{C1} and \eqref{C2} it is not hard to check  that the new function $f^{*}(r):=\frac{f(r)}{r^{m}}$ satisfies conditions \eqref{C1} and \eqref{C2} with $s$ replaced by $s-m.$ 
Indeed, \eqref{C1} (with $s$ replaced by $s-m$) follows since 
$
f^{*}(xy)=\frac{f(xy)}{(xy)^m}\asymp \frac{x^{s}f(y)}{x^{m}y^{m}} 
={x^{s-m}}f^{*}(y),$
 and 
\eqref{C2} (with $s$ replaced by $s-m$) follows since 
$$\frac{rf^{*'}(r)}{f^{*}(r)}=\frac{r}{f^{*}(r)}[r^{-m}f'(r)-mr^{-m-1}f(r)]=\left[r\frac{f'(r)}{f(r)}-m\right] \to ( s-m)\text{ as } r\to 0.$$ 

Now consider the product measure $\nu=\mu \times m_{\R^{m}},$ where $m_{\R^{m}}$ is the canonical Lebesgue measure on $\R^{m}$ and let $\pi_{1}$ and $\pi_{2}$ be the natural projections from $\R^{mn+m}$ to $\R^{mn}$ and $\R^{m},$ respectively. 

For any fixed integer $N\geq1,$ let $V_{S,\varepsilon}=W_{S,\varepsilon}\setminus \bigcup_{k=N}^{S-1}W_{k,\varepsilon}$ and $\widehat{V}_{S,\varepsilon}=\{(A,\bb)\in \widehat {W}_{S,\varepsilon} : A\in V_{S,\varepsilon}\}.$ Then $\nu(\bigcup_{S\geq N}\widehat{W}_{S,\varepsilon})=\nu(\bigcup_{S\geq N}\widehat{V}_{S,\varepsilon})\geq 1-2\varepsilon$, see \cite [p.21]{KiKi_20}.

Since $N\geq 1$ is arbitrary, we have $ \nu(\limsup_{S\to\infty}\widehat{W}_{S,\varepsilon})\ge 1-2\epsilon.$
For an arbitrary ball $B\subseteq \R^{mn+m}$ of sufficiently small radius $r(B),$
we have 
\begin{equation*}
\nu(B)\leq \mu(\pi_{1}(B))\times m_{\R^{m}}(\pi_{2}(B))\ll \frac{f(r(B))}{\eta},
\end{equation*}
where the implied constant does not depend on $B$ or $\eta.$
By using the Mass Distribution Principle i.e. Lemma \ref{MDP} and the Transference Lemma i.e. Lemma \ref{Transf}, we have
\begin{equation*}
\mathcal{H}^{f}(  \widehat{D}_{m, n}(\psi)^{c}       )\geq \mathcal{H}^{f}(  \limsup_{S\to\infty}\widehat{W}_{S,\varepsilon}) \gg  (1-2\varepsilon)\eta,  \end{equation*}
and by letting $\eta\to\infty$ we obtain the desired result.
\end{proof}

\subsection{Local ubiquity for $W_{\bb,\varepsilon}$}

We will use the idea of local ubiquity for $W_{\bb,\varepsilon}$ to prove the divergent part of Theorem \ref{BST2}. Following  \cite{KiKi_20} we define 
\begin{equation}\label{Ubiq1}
\varepsilon(\bb)=\min_{1\leq j\leq m, \ |b_{j}|_{\Z}>0}\frac{|b_{j}|_{\Z}}{4},
\end{equation}
for $\bb=(b_{1},\cdots,b_{m})\in\R^{m}\setminus\Z^{m}.$ 
 Note that $\varepsilon(\bb)>0$ is due to the fact that $\bb\in\R^{m}\setminus\Z^{m}.$  
 
The following lemma is used when we count the number of integral vectors $\zz\in\Z^{m}$ such that 
\begin{equation}\label{Ubiq2}
\left|\bb\cdot\zz\right|_{\Z}\leq \varepsilon(\bb).
\end{equation} 
\begin{lem}[{\cite[Lemma 4.4]{KiKi_20}}]
For $\bb=(b_{1},\cdots,b_{m})\in\R^{m}\setminus\Z^{m},$ let $\varepsilon(\bb)$ be as in \eqref{Ubiq1} and $1\leq i\leq m$ be an index such that $\varepsilon(\bb)=\frac{|b_{i}|_{\Z}}{4}.$ Then, for any $\xx\in\Z^{m},$ at most one of $\xx$ and $\xx+\eee_{i}$ satisfies \eqref{Ubiq2} where $\eee_{i}$ denotes the vector with a $1$ in the $ith$ coordinate and $0$'s elsewhere.
 \end{lem}

For a fixed
$\bb\in \R^{m}\setminus\Z^{m},$ let $\varepsilon_{0}:=\varepsilon(\bb),$ $\Psi_{0}:=\Psi_{\varepsilon_{0}}$ and $\Psi(h)=\frac{\Psi_{0}(h)}{h}.$ 
With notions in the Subsection \ref{US}, which are defined for the ubiquitous system construction,
let 
\begin{equation}
J:=\{ (\xx,\yy)\in\Z^{m}\times \Z^{n}: \|\yy\|\leq m\|\xx\| \text{ and } |\bb\cdot \xx|_{\Z}>\varepsilon_{0}\} \text{ and }\end{equation}
\begin{equation}
\text{ for }  \kappa:= (\xx,\yy)\in J \text{ denote } \beta_{\kappa}:=\|\xx\| \text{ and }R_{\kappa}:=\{A\in[0,1]^{mn}:A^{t}\xx=\yy \}.
\end{equation}

Note that $W_{\bb,\varepsilon_{0}}\subset\Lambda(\Psi)$ and the family $\mathcal R$ of resonant sets $R_{\kappa}$ consists of $(m-1)n$-dimensional, rational affine subspaces.

By Lemma \ref{BS4.2}, now we assume that the divergence part of Theorem \ref{BST2} is satisfied. 
Then we can find a strictly increasing sequence of positive integers $\{h_{i}\}_{i\in\N}$ such that 
\begin{equation}\sum\limits_{h_{i-1}<h\leq h_{i}}^{\infty}h^{m+n-1}\Big(\frac{\Psi_{0}(h)}{h}     \Big)^{-n(m-1)}f \Big(\frac{\Psi_{0}(h)}{h}     \Big) >1
\end{equation}
and $h_{i}>2h_{i-1}.$ Put $\omega(h):=i^{\frac{1}{n}}$ if $h_{i-1}<h\leq h_{i}.$ Then

$$\sum\limits_{h=1}^{\infty}h^{m+n-1}\Big(\frac{\Psi_{0}(h)}{h}     \Big)^{-n(m-1)}f \Big(\frac{\Psi_{0}(h)}{h}     \Big) \omega(h)^{-n} =\infty.$$

For a constant $c>0,$ define the ubiquitous function $\rho_{c}:\R_{+}\to\R_{+}$
by 
\begin{equation} 
\rho_{c}(h)=
\begin{cases}
 ch^{-\frac{1+n}{n}}   \  & \text{ if }\quad m=1; \\[2ex] 
ch^{-\frac{m+n}{n}}\omega(h) \  & \text{ if } \quad m\geq 2.
\end{cases}
\end{equation}
Clearly the ubiquitous function is $2$-regular.

\begin{thm}[{\cite[Theorem 4.5]{KiKi_20}}]\label{KKTh4.5}The pair $(\mathcal R,\beta)$ is a locally ubiquitous system relative to $\rho=\rho_{c}$ for some constant $c>0.$ 
\end{thm}
\noindent \textbf{The divergent part of Theorem \ref {BST2}}. 

Assume that $(m-1)n<s\leq mn$ and $r^{-nm}f(r)\to\infty$ as $r\to 0.$ It follows from Theorem \ref{thm5} and Theorem \ref{KKTh4.5} that
\begin{equation*}
\mathcal{H}^{f}(  \widehat{D}^{\bb}_{m, n}(\psi)^{c}       )\geq \mathcal{H}^{f}( W_{\bb,\varepsilon_{0}})=\mathcal{H}^{f}(  [0,1]^{mn}).
\end{equation*}
Similar as in \cite{KiKi_20} here we have used the fact that the divergence and convergence of the sums 
\begin{equation*}
\sum_{N=1}^{\infty}2^{\kappa N}\mathcal F(2^{N}) \text{ and } \sum_{h=1}^{\infty}h^{\kappa-1}\mathcal F(h)\end{equation*}
coincide for any monotonic function $\mathcal F:\Z_{+}\to\Z_{+}$ and $\kappa\in\R.$ This completes the proof of the divergent part of Theorem \ref{BST2}.

%
%
%
%

\end{document}